\documentclass{amsart}
\usepackage[T1]{fontenc}
\usepackage[utf8]{inputenc}

\usepackage[calc]{datetime2}

\usepackage{eso-pic}
\AddToShipoutPictureFG{
    \AtPageLowerLeft{%
        \makebox[0.95\paperwidth][r]{%
            \raisebox{3\height}{%
                \mbox{{\textrm{\tiny v\DTMToday}}}
            }%
        }%
    }%
}

\usepackage{amssymb}
\newcommand{\ud}{\mathrm{d}}
\newcommand{\Ud}{\mathrm{D}}
\newcommand{\norm}[1]{\left\lVert#1\right\rVert}
\newcommand{\abs}[1]{\left\lvert#1\right\rvert}
\newcommand{\bm}[1]{\mathbf{#1}}
\newcommand{\R}[1]{\ifnumequal{#1}{1}{\mathbb{R}}{\mathbb{R}^{#1}}}
\newcommand{\Dt}{\frac{\ud}{\ud t}}

\newtheorem{theorem}{Theorem}
\newtheorem{lemma}{Lemma}[section]
\newtheorem{corollary}[lemma]{Corollary}

\theoremstyle{definition}
\newtheorem{definition}[lemma]{Definition}

\DeclareMathOperator{\Div}{div}
\allowdisplaybreaks

\def\rot{{\rm rot}\,}

\numberwithin{equation}{section}

\usepackage[bookmarks=false]{hyperref}
\hypersetup{
    unicode=true,          
    pdftoolbar=true,        
    pdfmenubar=true,        
    pdffitwindow=false,     
    pdfstartview={FitH},    
    pdftitle={Global regular axially-symmetric solutions to the Navier-Stokes equations with small swirl},    
    pdfauthor={BN&WZ},     
    pdfsubject={NSE},   
    pdfcreator={WhoAmI},   
    pdfproducer={WhoAmI}, 
    pdfkeywords={keyword1, key2, key3}, 
    pdfnewwindow=true,      
    colorlinks=true,       
    linkcolor=black,          
    citecolor=blue,        
    filecolor=magenta,      
    urlcolor=cyan           
}

\title[Axially symmetric solutions to NSE]{Global regular axially-symmetric solutions to the Navier-Stokes equations with small swirl}

\begin{document}

\author[B. Nowakowski]{Bernard Nowakowski}
\address{Military University of Technology\\Cybernetics Faculty\\Institute of Mathematics and Cryptology\\Warsaw\\Poland}
\email{bernard.nowakowski@wat.edu.pl}

\author[W.M. Zajaczkowski]{Wojciech M. Zajaczkowski}
\address{Military University of Technology\\Cybernetics Faculty\\Institute of Mathematics and Cryptology\\Warsaw\\Poland}
\address{Polish Academy of Sciences\\Institute of Mathematics\\Warsaw\\Poland}
\email{wz@impan.pl}


\begin{abstract}
    Axially symmetric solutions to the Navier-Stokes equations in a bounded cylinder are considered. On the boundary the normal component of the velocity and he angular components of the velocity and vorticity are assumed to vanish. If the norm of the initial swirl is sufficiently small, then the regularity of axially symmetric, weak solutions is shown. The key tool is a new estimate for the stream function in certain weighted Sobolev spaces.
\end{abstract}

\maketitle

\section{Introduction}\label{s1}

In this work we consider axially-symmetric solutions to the Navier-Stokes equations in bounded cylindrical domains $\Omega\subset\R3$ with the boundary $S:=\partial \Omega$.

To describe the problem we  transform the Cartesian coordinates $x=(x_1,x_2,x_3)$ into cylindrical coordinates by the relation
\begin{equation*}
    x_1 = r\cos\varphi,\quad x_2=r\sin\varphi,\quad x_3=z.
\end{equation*}
This relation determines the orthonormal basis $(\bar e_r, \bar e_\varphi, \bar e_z)$, where
\begin{equation*}
    \bar e_r=(\cos\varphi,\sin\varphi,0),\quad \bar e_\varphi=(-\sin\varphi,\cos\varphi,0),\quad \bar e_z=(0,0,1)
\end{equation*}
are unit vectors along the radial-, the angular-, and the $z$-axes, respectively.

Using this orthonormal basis we can decompose the velocity vector $\bm v$ as follows
\begin{equation*}
    \bm v = v_r(r,z,t)\bar e_r + v_\varphi(r,z,t)\bar e_\varphi + v_z(r,z,t)\bar e_z
\end{equation*}
For the vorticity vector $\boldsymbol \omega=\rot \bm v$ we have the expression
\begin{equation*}
    \boldsymbol \omega=-v_{\varphi,z}(r,z,t)\bar e_r+\omega_\varphi(r,z,t)\bar e_\varphi+\frac{1}{r}(rv_\varphi)_{,r}(r,z,t)\bar e_z.
\end{equation*}
Here $\omega_\varphi$ can be computed explicitly, i.e. $\omega_\varphi=v_{r,z}-v_{z,r}$.

Let $R, a > 0$. Then, we define
\[
    \Omega=\{x\in\R3\colon r<R,\ |z|<a\}
\]
and by $\partial\Omega = S_1\cup S_2$ we denote the boundary of $\Omega$, where
\[
    \begin{aligned}
        S_1 &= \left\{x\in\R3\colon r=R,\ |z|<a\right\},\\
        S_2 &= \left\{x\in\R3\colon r<R,\ z\in\{-a,a\}\right\}.
    \end{aligned}
\]
The system of equations we investigate reads
\begin{equation}\label{1.5}
    \left\{
    \begin{aligned}
        &\bm v_t + (\bm v\cdot\nabla)\bm  v - \nu\Delta \bm v + \nabla p = \bm f & &\text{in $\Omega^T = \Omega\times(0,T)$},\\
        &\Div \bm v = 0 & &\text{in $\Omega^T$},\\
        &\bm v\cdot\bar n = 0 & &\text{on $S^T = S\times(0,T)$},\\
        &\omega_\varphi=0 & &\text{on $S^T$},\\
        &v_\varphi=0 & &\text{on $S^T$},\\
        &\bm v\big\vert_{t=0} = \bm v_0 & &\text{in $\Omega\times \{t = 0\}$}
    \end{aligned}
    \right.
\end{equation}
where $\bar n$ is the unit outward normal to $S$ vector.

To present our main result we need to introduce the quantity
\begin{equation}
    u=rv_\varphi.
\label{1.13}
\end{equation}
It is called \emph{the swirl} and is a solution to the problem
\begin{equation}
    \left\{
    \begin{aligned}
        &u_{,t} + \bm v\cdot\nabla u-\nu\Delta u+{2\nu\over r}u_{,r}=rf_\varphi\equiv f_0 & &\text{in $\Omega$}, \\
        &u = 0 & &\text{on $S^T$},\\
        &u\big\vert_{t=0} = rv_\varphi(0)\equiv u(0) & &\text{in $\Omega \times \{t = 0\}$}.
    \end{aligned}
    \right.
\label{1.14}
\end{equation}

We have to emphasize that the boundary conditions \eqref{1.5}$_{3,4}$ were introduced by O.A. Lady\v zhenskaya in \cite{Ladyzenskaja:1968aa}. Condition \eqref{1.5}$_4$ is necessary for solvability of some initial-boundary value problems for $\omega_\varphi$ (see \eqref{1.16}$_2$).
\begin{theorem}[Main result]\label{t1}
    Fix $0 < r_0 < R$. Let
    \[
        \begin{aligned}
            D_1^2 &\equiv 3\norm{\bm f}_{L_{1(0,t;L_2(\Omega))}(\Omega^t)}^2 + 2\norm{\bm v(0)}_{L_2(\Omega)}^2 < \infty,\\
            D_2 &\equiv \norm{f_0}_{L_1(0,t;L_{\infty}(\Omega))}+\norm{u(0)}_{L_\infty(\Omega)} < \infty.
        \end{aligned}
    \]
    Let us introduce
    \begin{multline}\label{1.21}
        M(t) = c\left(\norm{\frac{f_\varphi}{r}}_{L_2(0,t;L_{\frac{6}{5}}(\Omega))}^2 + \norm{f_{\varphi}}_{L_4(\Omega^t)}^4 + \norm{\frac{\omega_\varphi(0)}{r}}_{L_2(\Omega)}^2 \right. \\
    \left. + \int_\Omega{v_\varphi^4(0)\over r^2}\ud x\right) + c{D_1^{10}D_2^8\over r_0^{16}} \equiv M'(t)+c{D_1^{10}D_2^8\over r_0^{16}}.
    \end{multline}

    Let
    \[
        \begin{aligned}
            \alpha(t,r_0) &= \norm{u}_{L_{\infty}(\Omega_{r_0}^t)}^2, \qquad \text{where $\Omega_{r_0}=\{x\in\Omega\colon r\le r_0\}$},\\
            M &= M(T),\\
            M' &= M'(T).
        \end{aligned}
    \]
    Assume that $\gamma > 1$ and $\alpha(t, r_0)$ is so small that
    \begin{multline*}
        \alpha(t,r_0) \leq c(\gamma - 1)M \\
        \cdot \left(\gamma M D_2^2 + D_1^2 (\gamma M)^2 + (\gamma M)^2\exp\left(c(\gamma M)^2\right)\left(\norm{\frac{v_\varphi(0)}{r}}_{L_3(\Omega)}^2 + \norm{\frac{f_\varphi}{r}}_{L_1(0,t;L_3(\Omega))}^2\right)\right)^{-1} \\
        \equiv \Phi(M).
    \end{multline*}
    Then
    \begin{equation}\label{eq1.25n}
        \norm{\frac{\omega_\varphi}{r}}_{L_{\infty}(0,t;L_2(\Omega))}^2 + \norm{\frac{\omega_\varphi}{r}}_{L_2(0,t;H^1(\Omega))}^2 \leq \gamma M.
    \end{equation}

    Consider now the case $r_0=R$, thus $\Omega_R=\Omega$. Suppose that
    \[
        \alpha(t,R)\le\norm{f_0}_{L_1(0,t;L_{\infty}(\Omega))} + \norm{u(0)}_{L_{\infty}(\Omega)} \equiv \Phi(M').
    \]
    Then
    \begin{equation}\label{1.25}
        \norm{\frac{\omega_\varphi}{r}}_{L_{\infty}(0,t;L_2(\Omega))}^2 + \norm{\frac{\omega_\varphi}{r}}_{L_2(0,t;H^1(\Omega))}^2 \leq \gamma M'.
    \end{equation}
\end{theorem}

One may wonder what is the difference between \eqref{eq1.25n} and \eqref{1.25}. Careful comparison shows that \eqref{eq1.25n} is obtained provided that $\alpha(t,r_0) = \norm{u}^2_{L_\infty(\Omega_{r_0}^t)}$ is sufficiently small in the neighborhood of $r = 0$. In \eqref{1.25} we do not need any smallness restrictions. This might suggest that we can take $r_0 = R$ and without any restrictions show the regularity of weak, axially symmetric solutions with non-vanishing $v_\varphi(0)$. Unsurprisingly, this is not true: \eqref{1.25} does not exist without obtaining \eqref{eq1.25n} first. We will see later in the proof that we approach certain integral differently when $r$ is close to $0$ and when $0< r_0 < r$, where $r_0$ is fixed. Unfortunately, as \eqref{1.21} shows, passing with $r_0 \to 0^+$ is not possible.

We should emphasize that Theorem \ref{t1} does not directly imply the regularity of weak solutions but we may quickly deduce it following the reasoning from Lemma \ref{lem2.9n}. Instead, we utilize one of many Serrin-type regularity criteria, e.g. \cite[Theorem 3.(ii)]{Chae:2002vv}, which states that if $\omega_\varphi \in L_\infty(0,t;L_2(\Omega))$, then a weak solution $\bm v$ to \eqref{1.5} is regular.
Inequality \eqref{1.25} yields exactly
\begin{equation*}
    \norm{ \omega_\varphi}_{L_{\infty}(0,t;L_2(\Omega))} \leq cM',
\label{1.26}
\end{equation*}
which for $\bm v'=(v_r,v_z)$ yields
\begin{equation}
    \norm{\bm v'}_{L_\infty(0,t;H^1(\Omega))}\leq cM',
\label{1.27}
\end{equation}
and eventually
\begin{equation}
    \norm{v_r}_{L_\infty(0,t;L_6(\Omega))} + \norm{v_z}_{L_\infty(0,t;L_6(\Omega))} \le cM'.
    \label{1.29}
\end{equation}
In light of \cite[Theorem 1]{Neustupa:2000wu} the above inequality also implies the regularity of a weak solution $\bm v$ to \eqref{1.5}. In fact, there are many auxiliary results that could be utilized here. For a brief summary of Serrin-type regularity criteria for axially symmetric solutions to the Navier-Stokes equations we refer the reader to the introductions in e.g. \cite{KPZ}, \cite{Chen:2017we} and \cite{Rencawowicz:2019uy}. Lots of regularity criteria in terms of angular component of the velocity or of the swirl were established in e.g. \cite{CFZ}, \cite{Ly}, \cite{W}, \cite{LZ}, \cite{KP}, \cite{NP}, \cite{ZZ}.


In general, the problem of regularity of weak solutions to the Navier-Stokes equations in $\R3$ has a long history. In 1968 it was shown independently by Ladyzhenskaya \cite{Ladyzenskaja:1968aa} and Ukhovskii et al. \cite{Ukhovskii:1968aa} that in class of axially symmetric solutions any weak solution is regular provided that $v_\varphi(0) = 0$. Shortly after Ladyzhenskaya wrote a book (\cite{L1}) which laid foundations for intensive research on regularity of weak solutions.

Before describing the steps of the proof of Theorem \ref{t1} let us briefly discuss recent results. In \cite{Lei:2017um} the case $\Omega = \R3$ is studied. Lei et al. show that if $\sup_{t \geq 0} \abs{u(r,z,t)} \sim O\left(\ln^{-2} r\right)$ (see Corollary 1.3), then $\bm v$ is global and regular axially symmetric solution to \eqref{1.5}$_{1,2,6}$. This is an improvement over Wei's result (see \cite{Wei:2016tn}), where $O\left(\ln^{-\frac{3}{2}} r\right)$ is needed. These two results were recently improved in \cite{Chen:2022un}, where the condition
\[
    \abs{u(r,z,t)} \leq N e^{-c\abs{\ln r}^\tau}
\]
implies the regularity of weak solutions. Here $0 < r \leq \frac{1}{4}$ and $\tau$ is any number from $(0,1)$, $c,N$ are some constants. Our result is somehow comparable -- \eqref{1.21} suggests that $\abs{u(r,z,t)} \sim e^{-\frac{1}{r^{16}}}$.

We have to emphasize that in papers \cite{Ly}, \cite{LZ}, \cite{ZZ} smallness condition looks very complicated and depends not only on the swirl but also on e.g. vorticity. In \cite{Zajaczkowski:2010wb} to prove the regularity of weak, axially symmetric solutions we assume either $v_r \in L_\infty(0,t;L_3(\Omega))$ or $\frac{v_r}{r} \in L_\infty(0,t; L_{\frac{3}{2}}(\Omega))$. In both cases some smallness conditions are needed but they depend explicitly on the constant from the Poincar\'e inequality.

To the best of our knowledge that are not that many results concerning the regularity of weak, axially symmetric solutions to the Navier-Stokes equations in bounded cylinders (see e.g. \cite{Zajaczkowski:2007tu}). Our main result is not only new but it also uses non-trivial weighted estimates for the stream functions. To explain this technique, we go back to \eqref{1.5} and following e.g. Ladyzhenskaya \cite{Ladyzenskaja:1968aa} or How et al. (see \cite{HL}) we rewrite it in the form
\begin{equation}\label{1.6}
    \left\{
        \begin{aligned}
            &v_{\varphi,t} + \bm v\cdot\nabla v_\varphi-\nu\bigg(\Delta-\frac{1}{r^2}\bigg)v_\varphi+\frac{1}{r}v_rv_\varphi=f_\varphi & &\text{in $\Omega^T$}, \\
            &\begin{split}\omega_{\varphi,t} + \bm v\cdot\nabla\omega_\varphi-\nu\left(\Delta-\frac{1}{r^2}\right)\omega_\varphi+\frac{1}{r}\left(v_\varphi^2\right)_{,z} \\
                + \frac{1}{r}v_r\omega_\varphi=F_\varphi\end{split} & &\text{in $\Omega^T$}, \\
            &-\left(\Delta-\frac{1}{r^2}\right)\psi=\omega_\varphi & &\text{in $\Omega^T$}, \\
             &v_\varphi = \omega_\varphi = \psi = 0 & &\text{on $S^T$}, \\
             &v_\varphi\big\vert_{t=0}=v_\varphi(0) & &\text{in $\Omega\times \{t = 0\}$}, \\
             &\omega_\varphi\big\vert_{t=0}=\omega_\varphi(0) & &\text{in $\Omega\times \{t = 0\}$}, \\
        \end{aligned}
    \right.
\end{equation}
where $F_\varphi=\rot \bm f\cdot\bar e_\varphi$ and $\psi$ is the stream function such that
\begin{equation}
    v_r=-\psi_{,z},\quad v_z=\frac{1}{r}(r\psi)_{,r}.
\label{1.9}
\end{equation}
We recall that in \eqref{1.6} and whenever cylindrical coordinates in this manuscript are used we have
\begin{equation}
    \nabla = \bar e_r \partial_r + \bar e_z \partial_z \qquad \text{and} \qquad \Delta=\partial_r^2+\frac{1}{r}\partial_r+\partial_z^2.
\label{1.8}
\end{equation}
To derive energy type estimates for he velocity we prefer \eqref{1.5}$_{1,2}$ in the form
\begin{equation}\label{1.11}
    \begin{aligned}
        &v_{r,t} + \bm v\cdot\nabla v_r - \nu\left(\Delta v_r - \frac{1}{r^2}v_r\right) - \frac{1}{r}v_\varphi^2 + p_{,r} = f_r,\\
        &v_{\varphi,t} + \bm v\cdot\nabla v_\varphi - \nu\left(\Delta v_\varphi - \frac{1}{r^2}v_\varphi\right) + \frac{1}{r}v_rv_\varphi = f_\varphi,\\
        &v_{z,t} + \bm v\cdot\nabla v_z - \nu\Delta v_z + p_{,z} = f_z,\\
        &(rv_r)_{,r} + (rv_z)_{,z} = 0,
    \end{aligned}
\end{equation}
Moreover, we have the following boundary
\begin{equation}\label{1.10}
    v_r\big\vert_{S_1} = 0, \quad v_z\big\vert_{S_2} = 0,\quad v_\varphi\big\vert_S=0, \quad v_{r,z}-v_{z,r}\big\vert_S=0,
\end{equation}
and initial conditions
\begin{equation*}
    v_r\big\vert_{t=0} = v_r(0),\qquad v_\varphi\big\vert_{t=0}=v_\varphi(0),\qquad v_z\big\vert_{t=0}=v_z(0)
\end{equation*}

It is also convenient to introduce the quantities
\begin{equation}
    u_1 = \frac{v_{\varphi}}{r},\quad \omega_1= \frac{\omega_{\varphi}}{r},\quad \psi_1 = \frac{\psi}{r},\quad f_1 = \frac{f_{\varphi}}{r},\quad F_1 = \frac{F_{\varphi}}{r}.
\label{1.15}
\end{equation}
Then, system \eqref{1.6} finally reads
\begin{equation}\label{1.16}
    \left\{
        \begin{aligned}
            &u_{1,t} + \bm v\cdot\nabla u_1 - \nu\left(\Delta u_1+\frac{2}{r}u_{1,r}\right) = 2u_1\psi_{1,z} + f_1 & &\text{in $\Omega^T$},\\
            &\omega_{1,t} + \bm v\cdot\nabla\omega_1-\nu\bigg(\Delta\omega_1+\frac{2}{r}\omega_{1,r}\bigg)=2u_1u_{1,z}+F_1 & &\text{on $\Omega^T$}, \\
            &-\Delta\psi_1-\frac{2}{r}\psi_{1,r}=\omega_1 & &\text{in $\Omega^T$}, \\
            &u_1 = \omega_1 = \psi_1 = 0 & &\text{on $S^T$}, \\
            &u_1\big\vert_{t=0} = u_1(0) & &\text{in $\Omega \times \{t = 0\}$}, \\
            &\omega_1\big\vert_{t=0} = \omega_1(0) & &\text{in $\Omega \times \{t = 0\}$}, \\
        \end{aligned}
    \right.
\end{equation}

Systems \eqref{1.16} and \eqref{1.6} are similar. Our main focus will be concentrated on $\int_{\Omega^t} \frac{v_r}{r} \frac{v_\varphi^4}{r^2}\, \ud x\ud t'$. To handle this integral we need estimates for solutions to both \eqref{1.16} and \eqref{1.6}. These estimates are presented in Sections \ref{s2}, \ref{s3} and \ref{s4}. Finally, in Section \ref{s5} we eventually combine them. Apart from various energy estimates we also need two non-trivial estimates in weighted Sobolev spaces for solutions to \eqref{1.15}$_3$ (see Corollaries \ref{cor1} and \ref{cor2}). Due to the order of the weight, we need to adjust the order of singularity of $\psi_1$ near $r = 0$. In Lemma \ref{lem2.11} we will see that $\psi_1 \sim O(1)$, thus $\psi_1 \notin H^3_0(\Omega)$ (see Section \ref{s2}). Therefore, we subtract from $\psi_1$ as much as it is needed for this difference to belong to $H^3_0(\Omega)$. This idea is motivated by Kondratev's work (see \cite{Ko67}) and discussed in a separate manuscript (see \cite{NZ}).

\section{Notation and auxiliary results}\label{s2}

First we introduce the function spaces

\begin{definition}\label{d2.1}
    Let $\Omega$ be a cylindrical axially symmetric domain with axis of symmetry inside. We use the following notation for Lebesgue and Sobolev spaces:
    \[
        \begin{aligned}
            &\norm{u}_{L_p(Q)}=\abs{u}_{p,Q},\quad \norm{u}_{L_p(Q^t)}=\abs{u}_{p,Q^t},\\
            &\norm{u}_{L_{p,q}(Q^t)}=\norm{u}_{L_q(0,t;L_p(Q))}=\abs{u}_{p,q,Q^t},
        \end{aligned}
    \]
    where $p,q\in[1,\infty]$, $Q$ replaces either $\Omega$ or $S$.
    \[
        \begin{aligned}
            &\norm{u}_{H^s(Q)}=\norm{u}_{s,Q},\qquad \text{where $H^s(Q)=W_2^s(Q)$},\\
            &\norm{u}_{W_p^s(Q)}=\norm{u}_{s,p,Q},\\
            &\norm{u}_{L_q(0,t;W_p^k(Q))}=\norm{u}_{k,p,q,Q^t},\qquad \norm{u}_{k,p,p,Q^t}=\norm{u}_{k,p,Q^t},
        \end{aligned}
    \]
    where $s,k\in\R1_+$.
\end{definition}

Finally, similarly to Definition 2.1 in \cite{NZ} we introduce weighted spaces $L_{p,\mu}(\Omega)$, $\mu\in\R1$, $p\in[1,\infty]$, with the norm
\[
    \norm{u}_{L_{p,\mu}(\Omega)}=\left(\int_\Omega|u|^pr^{p\mu}\ud x\right)^{\frac{1}{p}}
\]
and
\[
    \norm{u}_{H^k_\mu(\Omega)} = \left(\sum_{\abs{\alpha} \leq k} \int_\Omega \abs{\Ud^\alpha_{r,z} u(r,z)}^2r^{2(\mu + \abs{\alpha} - k)}\, r\ud r\ud z\right)^{\frac{1}{2}},
\]
where $\Ud^\alpha_{r,z} = \partial_r^{\alpha_1}\partial_z^{\alpha_2}$, $\abs{\alpha} = \alpha_1 + \alpha_2$, $\abs{\alpha} \leq k$, $\alpha_i \in \mathbb{N}_0 \equiv\{0,1,2,\ldots\}$, $i=1,2$, $k\in \mathbb{N}_0$ and $\mu \in \mathbb{R}$. In fact, we only use $H^3_0(\Omega)$ and $H^2_0(\Omega)$ and these symbols should not be mixed with Sobolev spaces with zero trace.

We use notation: r.h.s -- right-hand side, l.h.s. -- left-hand side.

By $c$ we denote generic constants. They are time-independent but they may depend on $R$. If a constant depends on a quantity $l$ and this dependence needs to be tracked we write $c(l)$. This means that $c(l) \sim c \cdot l$. Similarly $c\left(\frac{1}{l}\right) \sim \frac{c}{l}$.

\begin{lemma}[Hardy's inequality]\label{lemH}
    Suppose that $f\ge 0$, $p\ge 1$ and $r\neq 0$. Then
    \begin{equation*}
    \bigg(\int_0^\infty\left(\int_0^xf(y)\, \ud y\right)^px^{-r-1}\ud x\bigg)^{1/p}\le \frac{p}{r}\left(\int_0^\infty|yf(y)|^py^{-r-1}\,\ud y\right)^{1/p}.
    \end{equation*}
\end{lemma}

\begin{lemma}\label{l2.2}
    Let $\bm f\in L_{2,1}(\Omega^t)$, $\bm v(0)\in L_2(\Omega)$. Assume that $v_\varphi\big\vert_S=0$, $\bar n\cdot \bm v\big\vert_S=0$, $\omega_\varphi\big\vert_S=0$. Then, solutions to \eqref{1.5} satisfy the estimate
    \begin{multline}\label{2.1}
        \norm{\bm v(t)}_{L_2(\Omega)}^2 + \nu\int_{\Omega^t}\left(|\nabla v_r|^2+|\nabla v_\varphi|^2+|\nabla v_z|^2\right)\,\ud x\ud t' \\
        +\nu\int_{\Omega^t}\left({v_r^2\over r^2}+{v_\varphi^2\over r^2}\right)\,\ud x\ud t'\leq D_1^2
    \end{multline}
\end{lemma}

\begin{proof}
    Multiplying \eqref{1.11}$_1$ by $v_r$, \eqref{1.11}$_2$ by $v_\varphi$, \eqref{1.11}$_3$ by $v_z$, adding the results, integrating over $\Omega$ and using \eqref{1.10} yields
    \begin{multline}\label{2.2}
        \frac{1}{2}\Dt\int_\Omega\left(v_r^2+v_\varphi^2+v_z^2\right)\,\ud x-\nu\int_{S_1} v_{z,r}v_zdS_1-\nu\int_{S_2}v_{r,z}v_rdS_2 \\
        + \nu\int_\Omega\left(|\nabla v_r|^2+|\nabla v_\varphi|^2+|\nabla v_z|^2\right)\,\ud x+ \nu\int_\Omega\bigg({v_r^2\over r^2}+{v_\varphi^2\over r^2}\bigg)\,\ud x \\
        + \int_\Omega\left(-\frac{1}{r}v_\varphi^2v_r+\frac{1}{r}v_rv_\varphi^2\right)\,\ud x +\int_\Omega\left(p_{,r}v_r+p_{,z}v_z\right)\,\ud x \\
        = \int_\Omega\left(f_rv_r+f_\varphi v_\varphi+f_zv_z\right)\,\ud x.
    \end{multline}
    In view \eqref{1.10} the boundary terms in \eqref{2.2} vanish. The last term on the l.h.s. of \eqref{2.2} vanishes in virtue of \eqref{1.10} and the equation of continuity \eqref{1.11}$_4$.

    Using that $\abs{\bm v}^2=v_r^2+v_\varphi^2+v_z^2$, we rewrite \eqref{2.2} the form
    \begin{multline}\label{2.3}
        \frac{1}{2}\Dt \norm{\bm v}_{L_2(\Omega)}^2+\nu\int_\Omega\left(|\nabla v_r|^2 + |\nabla v_\varphi|^2+|\nabla v_z|^2\right)\,\ud x + \nu\int_\Omega\bigg({v_r^2\over r^2}+{v_\varphi^2\over r^2}\bigg)\,\ud x \\
        =\int_\Omega\left(f_rv_r+f_\varphi v_\varphi+f_zv_z\right)\,\ud x.
    \end{multline}
    Applying the H\"older inequality to the r.h.s. of \eqref{2.3} yields
    \begin{equation}\label{2.4}
        \Dt\norm{\bm v}_{L_2(\Omega)}\le\norm{\bm f}_{L_2(\Omega)},
    \end{equation}
    where we used that $\abs{\bm f}^2=f_r^2+f_\varphi^2+f_z^2$.

    Integrating \eqref{2.4} with respect to time implies
    \begin{equation}
        \norm{\bm v(t)}_{L_2(\Omega)}\le\norm{\bm f}_{L_{2,1}(\Omega^t)}+\norm{\bm v(0)}_{L_2(\Omega)}.
        \label{2.5}
    \end{equation}
    Integrating \eqref{2.3} with respect to time, using the H\"older inequality in the r.h.s. of \eqref{2.3} and using \eqref{2.5} we obtain
    \begin{multline*}\label{2.6}
        \frac{1}{2}\norm{\bm v(t)}_{L_2(\Omega)}^2 + \nu\int_{\Omega^t}\left(|\nabla v_r|^2+|\nabla v_\varphi|^2+|\nabla v_z|^2\right)\,\ud x\ud t' +\nu\int_{\Omega^t}\bigg({v_r^2\over r^2}+{v_z^2\over r^2}\bigg)\,\ud x\ud t' \\
        \leq \norm{\bm f}_{L_{2,1(\Omega^t)}}\left(\norm{\bm f}_{L_{2,1}(\Omega^t)} +\norm{\bm v(0)}_{L_2(\Omega)}\right)+\frac{1}{2}\norm{\bm v(0)}_{L_2(\Omega)}^2.
    \end{multline*}
    The above inequality implies \eqref{2.1} and concludes the proof.
\end{proof}
\goodbreak

\begin{lemma}\label{l2.3}
    Consider problem \eqref{1.14}. Assume that $f_0\in L_{\infty,1}(\Omega^t)$, $u(0)\in L_\infty(\Omega)$. Then
    \begin{equation}
        \norm{u(t)}_{L_\infty(\Omega)}\le D_2
        \label{2.7}
    \end{equation}
\end{lemma}

\begin{proof}
    Multiplying $\eqref{1.14}_1$ by $u|u|^{s-2}$, $s>2$ integrating over $\Omega$ and by parts and using that $u\big\vert_S=0$, we obtain
    \begin{multline}
        \frac{1}{s}\Dt \norm{u}_{L_s(\Omega)}^s+{4\nu(s-1)\over s^2}\norm{\nabla|u|^{s/2}}_{L_2(\Omega)}^2 + {\nu\over s}\int_\Omega\left(|u|^s\right)_{,r}\, \ud r  \ud z \\
        =\int_\Omega f_0u|u|^{s-2}\ud x,
        \label{2.8}
    \end{multline}
    where the last term of \eqref{2.8} equals $I \equiv \frac{\nu}{s} \int_{-a}^a \abs{u}^s\big\vert_{r=0}^{r=R}\, \ud z$. From \cite{LW} it follows that $u\big\vert_{r=0}=0$. Since $u\big\vert_{r=R}=0$ and using the boundary condition $v_\varphi\big\vert_{S} = 0$ we conclude that $I = 0$. Then, we derive from \eqref{2.8} the inequality
    \begin{equation}
        \Dt\norm{u}_{L_s(\Omega)}\le\norm{f_0}_{L_s(\Omega)}.
    \label{2.9}
    \end{equation}
    Integrating \eqref{2.9} with respect to time and passing with $s\to\infty$ we derive \eqref{2.7} from \eqref{2.9}. This ends the proof.
\end{proof}

\begin{lemma}\label{l2.4}
Let estimates \eqref{2.1} and \eqref{2.7} hold. Then
\begin{equation}
\norm{v_\varphi}_{L_4(\Omega^t)}\le D_1^{1/2}D_2^{1/2}.
\label{2.10}
\end{equation}
\end{lemma}

\begin{proof}
We have
\[
    \int_{\Omega^t}|v_\varphi|^4\,\ud x\ud t'=\int_{\Omega^t}r^2v_\varphi^2{v_\varphi^2\over r^2}\,\ud x\ud t' \le\norm{rv_\varphi}_{L_\infty(\Omega^t)}^2\int_{\Omega^t}{v_\varphi^2\over r^2}\,\ud x\ud t'\le D_2^2D_1^2.
\]
This implies \eqref{2.10} and concludes the proof.
\end{proof}

\begin{lemma}\label{l2.5}
    Let $\omega_1\in L_2(\Omega)$. Then solutions to \eqref{1.16}$_3$ satisfy
    \begin{equation}
        \norm{\psi_1}_{H^1(\Omega)}^2+\int_{-a}^a\psi_1^2(0)\,\ud z\le c\norm{\omega_1}_{L_2(\Omega)}^2,
    \label{2.11}
    \end{equation}
    where $\psi_1(0)=\psi_1|_{r=0}$. In addition, if $\omega_1\in L_{2,\mu}(\Omega)$, $\mu\in(0,1)$ then
    \begin{equation}
        \norm{\psi_1}_{L_{2,-\mu}(\Omega)}^2+\norm{\psi_1}_{H^1(\Omega)}^2+\int_{-a}^a \psi_1^2(0)\,\ud z\le c\norm{\omega_1}_{L_{2,\mu}(\Omega)}^2,
    \label{2.12}
    \end{equation}
    where $\psi_1(0)=\psi_1|_{r=0}$.
\end{lemma}

\begin{proof}
    Multiply \eqref{1.16}$_3$ by $\psi_1$, integrate over $\Omega$ and use boundary condition \eqref{1.16}$_4$. Then we obtain
    \begin{equation}
        \norm{\nabla\psi_1}_{L_2(\Omega)}^2-\int_\Omega\partial_r\psi_1^2\, \ud r\ud z= \int_\Omega\omega_1\psi_1\, \ud x.
    \label{2.13}
    \end{equation}
    Applying the H\"older inequality to the r.h.s. of \eqref{2.13}, using the Poincar\'e inequality and boundary condition \eqref{1.14}$_4$ we obtain \eqref{2.11}.

    Using weighted spaces we can estimate the r.h.s. of \eqref{2.13} by
    $$
    \norm{\omega_1}_{L_{2,\mu}(\Omega)}\norm{\psi_1}_{L_{2,-\mu}(\Omega)}.
    $$
    By the Hardy inequality (see Lemma \ref{lemH}) and $\mu\in(0,1)$, $r\le R$, we get
    $$
    \int_\Omega|\psi_1|^2r^{-2\mu}\ud x\le c\int_\Omega|\psi_{1,r}|^2r^{2-2\mu}\ud x\le cR^{2-2\mu}\int_\Omega|\nabla\psi_1|^2\,\ud x.
    $$
    Since $\mu \in (0,1)$ the bound $\int_\Omega \abs{\psi_1}^2 r^{-2\mu}\, \ud x < \infty$ does not imply $\psi_1\big\vert_{r=0} = 0$. Then \eqref{2.12} holds. This concludes the proof.
\end{proof}

\begin{lemma}\label{l2.6}
    Assume that $u_1(0)\in L_\infty(\Omega)$, $f_1, \psi_{1,z}\in L_1(0,t;L_\infty(\Omega))$. Then for solutions to \eqref{1.16} the following inequality
    \begin{equation}\label{2.14}
        \norm{u_1(t)}_{L_\infty(\Omega)} \leq \exp\left(\int_0^t\norm{\psi_{1,z}(t')}_{L_\infty(\Omega)}\, \ud t'\right) D_2.
    \end{equation}
    holds.
\end{lemma}

\begin{proof}
    Multiply $\eqref{1.16}_1$ by $u_1|u_1|^{s-2}$ and integrate over $\Omega$. Then we have
    \begin{equation}
        \frac{1}{s}\Dt\norm{u_1}_{L_s(\Omega)}^s+{4\nu(s-1)\over s^2}\norm{\nabla u_1^{s/2}}_{L_2(\Omega)}^2 =\int_\Omega\psi_{1,z}u_1^s\ud x+\int_\Omega f_1u_1^{s-1}\ud x.
    \label{2.15}
    \end{equation}
    Applying the H\"older inequality to the r.h.s. of \eqref{2.15} and simplifying we get
    \begin{equation}
        \Dt\norm{u_1}_{L_s(\Omega)}\leq \norm{\psi_{1,z}}_{L_\infty(\Omega)}\norm{u_1}_{L_s(\Omega)} + \norm{f_1}_{L_s(\Omega)}.
    \label{2.16}
    \end{equation}
    Integrating with respect to time yields
    \begin{multline}\label{2.17}
        \norm{u_1(t)}_{L_s(\Omega)}\\
        \leq\exp\bigg(\int_0^t\norm{\psi_{1,z}(t')}_{L_\infty(\Omega)} \,\ud t'\bigg)\left(\norm{f_1}_{L_1(0,t;L_s(\Omega))}+\norm{u_1(0)}_{L_s(\Omega)}\right).
    \end{multline}
    Passing with $s\to\infty$ we derive \eqref{2.14}. This concludes the proof.
\end{proof}

\begin{lemma}\label{lem2.11}
    Let $\psi_1$ be a solution to
    \begin{equation}\label{eq2.22}
        \left\{
            \begin{aligned}
                -&\Delta\psi_1 - \frac{2}{r}\psi_{1,r} = \omega_1 & &\text{in $\Omega$},\\
                 &\psi_1 = 0 & &\text{on $S$}.
            \end{aligned}
        \right.
    \end{equation}
    Suppose that $\omega_1 \in L_2(\Omega)$. Then, any solution $\psi_1$ to \eqref{eq2.22} satisfies
    \begin{equation}\label{eq2.23}
        \norm{\psi_1}_{2,\Omega} \leq c \abs{\omega_1}_{2,\Omega}.
    \end{equation}
\end{lemma}

\begin{proof}
    We start with rewriting \eqref{eq2.22}$_1$ in the form
    \[
        -\psi_{1,rr} - \psi_{1,zz} - \frac{3}{r}\psi_{1,r} = \omega_1.
    \]
    Multiplying this equality by $\frac{1}{r}\psi_{1,r}$ and integrating over $\Omega$ yields
    \begin{equation}\label{eq2.25}
        3\int_\Omega \abs{\frac{1}{r}\psi_{1,r}}^2\, \ud x = -\int_\Omega \psi_{1,rr} \frac{1}{r} \psi_{1,r}\, \ud x - \int_\Omega \psi_{1,zz} \frac{1}{r} \psi_{1,r}\, \ud x - \int_\Omega \omega_1 \frac{1}{r} \psi_{1,r}\, \ud x.
    \end{equation}
    The first term on the r.h.s. of \eqref{eq2.25} equals
    \[
        -\int_\Omega \psi_{1,rr}\psi_{1,r}\, \ud r \ud z = -\frac{1}{2} \int_\Omega \partial_r \psi_{1,r}^2\, \ud r \ud z = -\frac{1}{2} \int_{-a}^a \psi_{1,r}^2\big\vert_{r = 0}^{r = R}\, \ud z \equiv I_1.
    \]
    Integrating with respect to $z$ in the second term on the r.h.s. of \eqref{eq2.25} yields
    \[
        -\int_\Omega \left(\psi_{1,z}\psi_{1,r}\right)_{,z}\, \ud r \ud z + \int_\Omega \psi_{1,z}\psi_{1,rz}\, \ud r \ud z
    \]
    where the first term vanishes because $\psi_{1,r}\vert_{z \in \{-a,a\}} = 0$ and the second equals
    \[
        I_2 \equiv \frac{1}{2} \int_{-a}^a \psi_{1,z}^2\big\vert_{r = 0}^{r = R}\, \ud z.
    \]
    Using the boundary condition \eqref{eq2.22}$_2$ we obtain
    \[
        I_2 = -\frac{1}{2} \int_{-a}^a \psi_{1,z}^2\big\vert_{r = 0}\, \ud z.
    \]
    From \cite[Remark 4]{NZ1} we have
    \[
        \begin{aligned}
            \psi &= a_1(r,z,t)\big\vert_{r = 0} r + a_3(r,z,t)\big\vert_{r = 0}r^3 + o(r^4),\\
            \psi_1 &= a_1(r,z,t)\big\vert_{r = 0}  + a_3(r,z,t)\big\vert_{r = 0}r^2 + o(r^3),\\
        \end{aligned}
    \]
    thus
    \begin{equation}\label{eq2.26}
        \psi_{1,r}\big\vert_{r = 0} = 0.
    \end{equation}
    Using \eqref{eq2.26} in $I_1$ yields
    \[
        I_1 = -\frac{1}{2} \int_{-a}^a \psi_{1,r}^2\big\vert_{r = R}\, \ud z.
    \]
    Applying the H\"older and Young inequalities to the last term on the r.h.s in \eqref{eq2.25} and combining it with $I_1$ and $I_2$ we obtain
    \begin{equation}\label{eq2.27}
        \frac{1}{2}\int_\Omega \frac{1}{r^2}\psi_{1,r}^2\, \ud x + \frac{1}{2}\int_{-a}^a \psi_{1,r}^2\big\vert_{r = R}\, \ud z + \frac{1}{2} \int_{-a}^a \psi_{1,z}^2\big\vert_{r = 0}\, \ud z \leq c \abs{\omega_1}_{2,\Omega}^2.
    \end{equation}
    Since the last two termns on the l.h.s. are positive we conclude that
    \begin{equation}\label{eq2.28}
        \int_\Omega \frac{1}{r^2} \psi_{1,r}^2\, \ud x \leq c \abs{\omega_1}_{2,\Omega}^2.
    \end{equation}
    Now we can rewrite \eqref{eq2.22} in the form

    \begin{equation}\label{eq2.29}
        \left\{
            \begin{aligned}
                -&\Delta\psi_1 = \omega_1 + \frac{2}{r}\psi_{1,r} & &\text{in $\Omega$},\\
                 &\psi_1 = 0 & &\text{on $S$}
            \end{aligned}
        \right.
    \end{equation}
    and consider it as the Dirichlet problem for the Poisson equation. Thus
    \begin{equation}\label{eq2.30}
        \norm{\psi_1}_{2,\Omega} \leq c \abs{\omega_1}_{2,\Omega},
    \end{equation}
    where \eqref{eq2.28} was used. This ends the proof.
\end{proof}

\begin{lemma}\label{lem2.9n}
    Assume that $s \in (1,\infty)$. Suppose that $f\in L_1(0,t;L_s(\Omega))$ and $u_1(0) \in L_s(\Omega)$. Then
    \begin{equation*}
        \abs{u_1}_{s, \Omega} \leq \exp\left(cs\int_0^t \abs{\omega_1(t')}^2_{2,\Omega}\, \ud t'\right)\left(s\norm{f_1}_{L_1(0,t;L_s(\Omega))} + \norm{u_1(0)}_{L_s(\Omega)}\right).
    \end{equation*}
\end{lemma}
\begin{proof}
    In \eqref{2.15} we integrate by parts, use the boundary conditions \eqref{1.14}$_4$ and apply the H\"older and Young inequalities
    \begin{multline}\label{5.8}
        \frac{1}{s}\Dt\abs{u_1}_{s,\Omega}^s+{4(s-1)\nu\over s^2}\int_\Omega \abs{\nabla|u_1|^{s/2}}^2\,\ud x \\
        \leq\epsilon\abs{\partial_zu_1^{s/2}}_{2,\Omega}^2 + \frac{c}{\varepsilon} \abs{\psi_1}_{\infty,\Omega}^2\abs{u_1}_{s,\Omega}^s + \abs{f_1}_{s,\Omega}\abs{u_1}_{s,\Omega}^{s-1}.
    \end{multline}
    For sufficiently small $\varepsilon$ we get
    \begin{equation}
        \frac{1}{s}\Dt\abs{u_1}_{s,\Omega}^s\leq cs\abs{\psi_1}_{\infty,\Omega}^2\abs{u_1}_{s,\Omega}^s + \abs{f_1}_{s,\Omega}\abs{u_1}_{s,\Omega}^{s-1}.
        \label{5.9}
    \end{equation}
    Hence, we have
    \[
        \Dt\abs{u_1}_{s,\Omega}\leq cs \abs{\psi_1}_{\infty,\Omega}^2\abs{u_1}_{s,\Omega} + \abs{f_1}_{s,\Omega}.
    \]
    Since $\epsilon = \frac{2(s-1)\nu}{s^2}$, then $\frac{c}{\epsilon} = \frac{cs^2}{2(s-1)\nu} \leq cs$. Integrating with respect to time yields
    \begin{equation}
        \abs{u_1}_{s,\Omega}\leq \exp\left(cs\int_0^t\abs{\psi_1(t')}_{\infty,\Omega}^2\,\ud t'\right)\left( \abs{u_1(0)}_{s,\Omega} + \abs{f_1}_{s,1,\Omega^t}\right).
        \label{2.20}
    \end{equation}
    Using Lemma \ref{lem2.11}
    \[
        \abs{\psi_1}_{\infty,\Omega}\le c\norm{\psi_1}_{2,\Omega}\le c\abs{\omega_1}_{2,\Omega}
    \]
    we obtain
    \begin{equation}\label{2.21}
        \abs{u_1}_{s,\Omega}\le\exp\left(cs\int_0^t\abs{\omega_1(t')}_{2,\Omega}^2\,\ud t'\right)\left( \abs{u_1(0)}_{s,\Omega} + \abs{f_1}_{s,1,\Omega^t}\right).
    \end{equation}
    This concludes the proof.
\end{proof}

\begin{corollary}[Theorem 1.3 in \cite{NZ}]\label{cor1}
    Suppose that $\psi_1$ is a weak solution to \eqref{1.16}$_{3,4}$. Let $\omega_1\in L_2(\Omega)$ and introduce
    \[
        \chi(r,z) = \int_0^r\psi_{1,\tau}(1+K(\tau))\,\ud\tau,
    \]
    where $K(\tau)$ is any smooth function with a compact support such that
    \[
        \lim_{r\to 0^+} \frac{K(r)}{r^2} = c_0<\infty.
    \]
    Then
    \begin{equation*}
        \norm{\psi_1 - \psi_1(0) - \chi}_{L_2(-a,a;H_0^2(0,R))}^2 + \norm{\psi_{1,zr}}_{L_2(\Omega)}^2 \\
        +\norm{\psi_{1,zz}}_{L_2(\Omega)}^2 \leq c\norm{\omega_1}_{L_2(\Omega)}^2,
    \end{equation*}
\end{corollary}

\begin{corollary}[Theorem 1.4 in \cite{NZ}]\label{cor2}
    Let $\psi_1$ be a weak solution to \eqref{1.14}$_{3,4}$. Let $\omega_1\in H^1(\Omega)$. Then
    \begin{multline*}
        \int_{\mathbb{R}}\norm{\psi_1 - \psi_1(0) - \eta}_{H_0^3(\mathbb{R}_+)}^2\, \ud z + \int_{\mathbb{R}}\int_{\mathbb{R}_+} \left(\abs{\psi_{1,zzz}}^2 + \abs{\psi_{1,zzr}}^2 + \abs{\psi_{1,zz}}^2\right)\, r\ud r \ud z \\
        \leq c\norm{\omega_1}_{H^1(\Omega)}^2,
    \end{multline*}
    where
    \[
        \eta(r,z) = - \int_0^r(r-\tau)\bigg({3\over r}\psi_{1,\tau}+\psi_{1,zz}+\omega_1\bigg) (1+K(\tau))\,\ud\tau
    \]
    and $K$ is the same as in Corollary \ref{cor1}.
\end{corollary}

\section{Estimate for $\omega_1$}\label{s3}

\begin{lemma}\label{l3.1}
    Assume that $\omega_1(0)\in L_2(\Omega)$, $u_1\in L_4(\Omega^t)$, $F\in L_{6/5,2}(\Omega^t)$, $t\le T$. Then the following inequality holds
    \begin{multline}\label{3.1}
        \frac{1}{2}\int_\Omega\omega_1^2\,\ud x + \frac{\nu}{2}\int_{\Omega^t}\abs{\nabla\omega_1}^2\,\ud x\ud t' + \nu \int_0^t\int_{-a}^a\omega_1^2\big\vert_{r=0}\, \ud z \ud t'\\
        \leq\frac{1}{\nu}\int_{\Omega^t}u_1^4\,\ud x\ud t'+c\abs{F_1}_{6/5,2,\Omega^t}^2+\int_\Omega \omega_1^2(0)\,\ud x.
    \end{multline}
\end{lemma}

\begin{proof}
    Multiply \eqref{1.16}$_2$ by $\omega_1$, integrate over $\Omega$, integrate by parts. Next, integration with respect to time implies \eqref{3.1}. This ends the proof.
\end{proof}

\section{Estimate for the angular component of velocity}\label{s4}

Consider problem \eqref{1.6}

\begin{lemma}\label{l4.1}
    Assume that $f_\varphi\in L_2(\Omega^t)$, $v_\varphi(0)\in L_{4,-1/2}(\Omega)$,
    \begin{equation}
        \abs{\int_{\Omega^t}{v_r\over r}{v_\varphi^4\over r^2}\,\ud x\ud t'}<\infty.
    \label{4.1}
    \end{equation}
    Then, any solution to \eqref{1.6} satisfy
    \begin{multline}\label{4.2}
        \frac{1}{4}\int_\Omega\frac{v_\varphi^4}{r^2}\, \ud x + \frac{3\nu}{4}\int_{\Omega^t}\abs{\nabla{\frac{v_\varphi^2}{r}}}^2\,\ud x \ud t' + \frac{\nu}{2}\int_{\Omega^t}\abs{\frac{v_\varphi}{r}}^4\, \ud x \ud t' \\
        \leq \frac{3}{2}\int_{\Omega^t}\frac{v_r}{r}\frac{v_\varphi^4}{r^2}\, \ud x \ud t' + \frac{27}{4\nu^3}\int_{\Omega^t}f_\varphi^4r^4\, \ud x \ud t' + \frac{1}{4}\int_\Omega\frac{v_\varphi^4(0)}{r^2}\, \ud x.
    \end{multline}
\end{lemma}

\begin{proof}
    Multiply $\eqref{1.6}_1$ by ${v_\varphi^3\over r^2}$ (see expansion \eqref{4.4} of $v_\varphi$ near the axis of symmetry) and integrate over $\Omega$. Then we have
    \begin{multline}\label{4.3}
        \frac{\ud}{4\,\ud t}\int_\Omega{v_\varphi^4\over r^2}\ud x + \int_\Omega \bm v\cdot\nabla v_\varphi \frac{v_\varphi^3}{r^2}\ud x - \nu\int_\Omega\Delta v_\varphi\frac{v_\varphi^3}{r^2}\ud x + \nu\int_\Omega\frac{v_\varphi^4}{r^4}\ud x \\
        + \int_\Omega \frac{v_r}{r}\frac{v_\varphi^4}{r^2}\ud x = \int_\Omega f_\varphi\frac{v_\varphi^3}{r^2}\ud x.
    \end{multline}
    The second term in \eqref{4.3} equals
    \[
        \frac{1}{4}\int_\Omega \bm v\cdot\nabla v_\varphi^4r^{-2}\ud x = \frac{1}{4}\int_\Omega \bm v\cdot\nabla\left(v_\varphi^4r^{-2}\right)\,\ud x + \frac{1}{2}\int_\Omega v_rv_\varphi^4r^{-3}\ud x =\frac{1}{2}\int_\Omega{v_r\over r}{v_\varphi^4\over r^2}\ud x,
    \]
    where we used that $\bm v\cdot\bar n\big\vert_S=0$ and $\Div \bm v = 0$.

    Integrating by parts in the third term on the l.h.s. of \eqref{4.3} yields
    \begin{multline*}
        \int_\Omega\nabla v_\varphi\nabla v_\varphi^3r^{-2}\ud x+\int_\Omega\nabla v_\varphi v_\varphi^3\nabla r^{-2}\ud x
        = 3\int_\Omega v_\varphi^2|\nabla v_\varphi|^2r^{-2}\ud x-2\int_\Omega v_{\varphi,r}v_\varphi^3r^{-3}\ud x\\
        =\frac{3}{4}\int_\Omega|\nabla v_\varphi^2|^2r^{-2}\ud x-\frac{1}{2}\int_\Omega\partial_rv_\varphi^4r^{-2}\, \ud r \ud z\\
        =\frac{3}{4}\int_\Omega\abs{\frac{\nabla v_\varphi^2}{r}}^2\,\ud x-\frac{1}{2}\int_\Omega\partial_r\left(v_\varphi^4r^{-2}\right)\, \ud r\ud z-\int_\Omega v_\varphi^4r^{-3}\, \ud r  \ud z\equiv I.
    \end{multline*}
    The first term in $I$ equals
    \begin{multline*}
        \frac{3}{4}\int_\Omega\abs{\nabla{v_\varphi^2\over r}-v_\varphi^2\nabla\frac{1}{r}}^2\,\ud x=\frac{3}{4}\int_\Omega\abs{\nabla{v_\varphi^2\over r}}^2\,\ud x -\frac{3}{2}\int_\Omega\nabla{v_\varphi^2\over r}\cdot v_\varphi^2\nabla\frac{1}{r}\ud x \\
        + \frac{3}{4}\int_\Omega\abs{v_\varphi^2\nabla\frac{1}{r}}^2\,\ud x =\frac{3}{4}\int_\Omega\abs{\nabla{v_\varphi^2\over r}}^2\,\ud x + \frac{3}{2}\int_\Omega\partial_r{v_\varphi^2\over r}{v_\varphi^2\over r^2}\ud x+\frac{3}{4} \int_\Omega\abs{{v_\varphi\over r}}^4\,\ud x \equiv J.
    \end{multline*}
    The middle term in $J$ can be written in the form
    \[
        \frac{3}{4}\int_\Omega\partial_r{v_\varphi^4\over r^2}\,\ud r  \ud z = \frac{3}{4}\int_{-a}^a{v_\varphi^4\over r^2}\bigg\vert_{r=0}^{r=R}\, \ud z\equiv L.
    \]
    From \cite[Remark 4]{NZ1} it follows that $v_\varphi$ behaves as
    \begin{equation}\label{4.4}
        v_\varphi = a_1(r,z,t)\big\vert_{r = 0}r + a_3(r,z,t)\big\vert_{r = 0} r^3 + o(r^4), \qquad r \approx 0,
    \end{equation}
    for some functions $a_1$ and $a_3$. Since $v_\varphi\big\vert_{r=R}=0$ the second terms in $I$ and $L$ vanish.

    Using the above calculations in \eqref{4.3} yields
    \begin{multline}\label{4.5}
        \frac{1}{4}\Dt\int_\Omega{v_\varphi^4\over r^2}\ud x+\frac{3}{4}\nu\int_\Omega\abs{\nabla{v_\varphi^2\over r}}^2\, \ud x + \frac{3}{4}\nu\int_\Omega{v_\varphi^4\over r^4}\ud x + \frac{3}{2}\int_\Omega{v_r\over r}{v_\varphi^4\over r^2}\ud x \\
        =\int_\Omega f_\varphi{v_\varphi^3\over r^2}\ud x.
    \end{multline}
    Applying the H\"older and Young inequalities to the r.h.s. of \eqref{4.5} and integrating the result with respect to time imply \eqref{4.2}. This concludes the proof.
\end{proof}

\section{Global estimate}\label{s5}

Multiplying \eqref{3.1} by $\frac{\nu^2}{4}$ and adding \eqref{4.2} we obtain
\begin{multline}\label{5.1}
    {\nu^2\over 8}\int_\Omega\omega_1^2(t)\,\ud x + \frac{\nu^3}{8}\int_{\Omega^t}|\nabla\omega_1|^2\,\ud x\ud t'+ \frac{1}{2}\int_\Omega{v_\varphi^4(t)\over r^2}\ud x \\
    + \frac{3\nu}{4}\int_{\Omega^t}\bigg|\nabla{v_\varphi^2\over r}\bigg|^2\,\ud x\ud t' + \frac{\nu}{4}\int_{\Omega^t}\bigg|{v_\varphi\over r}\bigg|^4\, \ud x \,\ud t' \leq \frac{3}{2}\bigg|\int_{\Omega^t}{v_r\over r}{v_\varphi^4\over r^2}\,\ud x\ud t'\bigg| \\
    + c\bigg(\abs{F_1}_{6/5,2,\Omega^t}^2+\abs{\omega_1(0)}_{2,\Omega}^2 +\int_{\Omega^t}r^4f_\varphi^4\,\ud x\ud t'+\int_\Omega{v_\varphi^4(0)\over r^2}\ud x\bigg).
\end{multline}
Therefore, we have to estimate the first term on the r.h.s. of \eqref{5.1}. To examine it we introduce the sets
\begin{equation}\label{5.2}
    \Omega_{r_0}=\{x\in\Omega\colon r\le r_0\}, \qquad \bar\Omega_{r_0}=\{x\in\Omega\colon r\ge r_0\},
\end{equation}
where $r_0>0$ is given.

We write the first term on the r.h.s. of \eqref{5.1} in the form
\begin{equation}\label{5.3}
    \int_{\Omega^t}{v_r\over r}{v_\varphi^4\over r^2}\,\ud x\ud t' = \int_{\Omega_{r_0}^t}{v_r\over r}{v_\varphi^4\over r^2}\,\ud x\ud t' + \int_{\bar\Omega_{r_0}^t}{v_r\over r}{v_\varphi^4\over r^2}\,\ud x\ud t' \equiv I+J.
\end{equation}

\begin{lemma}\label{l5.1}
    Under the assumptions of Lemmas \ref{l2.2} and \ref{l2.4} we have
    \begin{equation}\label{5.4}
        \abs{J}\leq\varepsilon_1\int_{\bar\Omega_{r_0}^t}\abs{\partial_z\frac{v_\varphi^2}{r}}^2\,\ud x\ud t' + \epsilon_2\sup_t\abs{\psi_{,xx}}_{2,\Omega}^2 + c\left(\frac{1}{\epsilon_1},\frac{1}{\epsilon_2}\right){D_1^{10}D_2^8\over r_0^{16}}.
    \end{equation}
\end{lemma}

\begin{proof}
    Since ${v_r\over r}=-\psi_{1,z}$ we have
    \begin{multline*}
        |J| = \bigg|\int_{\bar\Omega_{r_0}^t}\psi_{1,z}{v_\varphi^4\over r^2}\,\ud x\ud t'\bigg| \\
        \leq\varepsilon_1\int_{\bar\Omega_{r_0}^t}\bigg|\partial_z{v_\varphi^2\over r}\bigg|^2\,\ud x\ud t'+c\left(\frac{1}{\epsilon_1}\right)\int_{\bar\Omega_{r_0}^t}\psi_1^2{v_\varphi^4\over r^2}\,\ud x\ud t'\equiv J_1.
    \end{multline*}
    In view of Lemma \ref{l2.4} the second term in $J_1$ is bounded by
    \[
    \frac{1}{r_0^4}\int_{\bar\Omega_{r_0}^t}\psi^2v_\varphi^4\,\ud x\ud t'\le{D_1^2D_2^2\over r_0^4}\sup_{\bar\Omega_{r_0}^t}\psi^2\le c{D_1^2D_2^2\over r_0^4}\sup_t\abs{\psi_{,xx}}_{2,\Omega}^{\frac{3}{2}}\abs{\psi}_{2,\Omega}^{\frac{1}{2}}\equiv J_2.
    \]
    Note that all consideration are either \emph{a priori} or performed for regular, local solutions. Then, derivation of regular, global solutions can be achieved by extension with respect to time. Since $\psi$ is a solution to the problem
    \[
        \left\{
            \begin{aligned}
                &-\Delta\psi+{\psi\over r^2}=\omega & &\text{in $\Omega^T$}, \\
                &\psi = 0 & &\text{ on $S^T$},
            \end{aligned}
        \right.
    \]
    we have
    \[
        \int_\Omega|\nabla\psi|^2\, \ud x + \int_\Omega{\psi^2\over r^2}\, \ud x \leq \int_\Omega \bm v'^2\, \ud x \leq cD_1^2.
    \]
    Then $J_2$ is bounded by
    \[
        J_2\le\varepsilon_2\sup_t\abs{\psi_{,xx}}_{2,\Omega}^2 + c\left(\frac{1}{\epsilon_2}\right) \frac{D_1^8D_2^8}{r_0^{16}}D_1^2.
    \]
    Using estimates for $J_1$ and $J_2$ we derive \eqref{5.4}. This ends the proof.
\end{proof}

\begin{lemma}\label{l5.2}
    Let the assumptions of Lemma \ref{l2.2} hold. Additionally, assume that $v_\varphi(0)\in L_4(\Omega)$, $u\in L_\infty(\Omega^t)$, $\abs{u}_{\infty,\Omega^t}\le D_2$. Then $I$ from \eqref{5.3} satisfies
    \begin{multline}\label{5.5}
        \abs{I} \leq\epsilon_3\abs{\partial_z\frac{v_\varphi^2}{r}}_{2,\Omega_{r_0}^t}^2 + c\left(\frac{1}{\epsilon_3}\right) \abs{u}_{\infty,\Omega_{r_0}^t}^2\left(D_2^2\abs{\nabla\omega_1}_{2,\Omega^t}^2 \vphantom{ \left(\abs{u_1(0)}^2 + \abs{f_1}_{3,1,\Omega^t}^2\right) }\right. \\
        + \left.\abs{\omega_1}_{2,\infty,\Omega^t}^4D_1^2 + \left(\abs{u_1(0)}_{3,\Omega_{r_0}}^2 + \abs{f_1}_{3,1,\Omega^t}^2\right) \abs{\omega_1}_{2,\Omega^t}^2\exp \left(c\abs{\omega_1}_{2,\Omega^t}^2\right)\right).
    \end{multline}
\end{lemma}

\begin{proof}
    We have
    $$
    |I|\le\varepsilon_3\int_{\Omega_{r_0}^t}\bigg|\partial_z{v_\varphi^2\over r}\bigg|^2\,\ud x\ud t'+c(1/\varepsilon_3)\int_{\Omega_{r_0}^t}\psi_1^2{v_\varphi^4\over r^2}\,\ud x\ud t'\equiv I_1+I_2.
    $$
    We estimate $I_2$ by
    \begin{multline*}
        I_2 \le\int_{\Omega_{r_0}^t}|\psi_1-\psi_1(0) - \eta|^2{v_\varphi^4\over r^2}\,\ud x\ud t'+ \int_{\Omega_{r_0}^t}|\eta|^2{v_\varphi^4\over r^2}\,\ud x\ud t' +\int_{\Omega_{r_0}^t}|\psi_1(0)|^2{v_\varphi^4\over r^2}\,\ud x\ud t'\\
        \equiv I_2^1+I_2^2+I_2^3,
    \end{multline*}
    where $\psi_1(0)=\psi_1\big\vert_{r=0}$ and $\eta$ is defined in Corollary \ref{cor2}. Using this Corollary we have
    \begin{multline*}
        I_2^1 =\int_{\Omega_{r_0}^t}{|\psi_1-\psi_1(0) - \eta|^2\over r^6}{r^6v_\varphi^4\over r^2}\,\ud x\ud t' \\
        \le c\sup_{\Omega_{r_0}^t}\abs{u}^4\int_{\Omega_{r_0}^t}{|\psi_1-\psi_1(0) - \eta|^2\over r^6}\,\ud x\ud t'\le c\sup_{\Omega_{r_0}^t}|u|^4\abs{\nabla\omega_1}_{2,\Omega^t}^2.
    \end{multline*}
    Consider $I_2^3$,
    \begin{multline*}
        I_2^3 \le\sup_{\Omega_{r_0}}|\psi_1(0)|^2\sup_{\Omega_{r_0}^t}|u| \int_{\Omega_{r_0}^t}\bigg|{v_\varphi\over r}\bigg|^3\,\ud x\ud t'\\
        =\sup_{\Omega_{r_0}}|\psi_1(0)|^2\sup_{\Omega_{r_0}^t}|u|\int_{\Omega_{r_0}^t} {v_\varphi^2\over r^2}\bigg|{v_\varphi\over r}\bigg|\,\ud x\ud t' \\
        \le\sup_{\Omega_{r_0}^t}\abs{\psi_1}^2\sup_{\Omega_{r_0}^t}|u|\abs{\frac{v_\varphi}{r}}_{4,\Omega_{r_0}^t}^2\abs{{v_\varphi\over r}}_{2,\Omega_{r_0}^t} \\
        \leq \varepsilon\abs{{v_\varphi\over r}}_{4,\Omega_{r_0}^t}^4 + c\left(\frac{1}{\epsilon}\right)\sup_t\abs{\omega_1}_{2,\Omega}^4 \sup_{\Omega_{r_0}^t}|u|^2D_1^2,
    \end{multline*}
    where we used Lemmas \ref{l2.2} and \ref{lem2.11}.

    Consider $I_2^2$. To simplify presentation we express $\eta$ in the short form
    $$
        \eta = \int_0^r(r-\tau)f(\tau)d\tau,
    $$
    where $f$ replaces $\left(\frac{3}{r}\psi_{1,r} + \psi_{1,zz} + \omega_1\right)\left(1 + K(r)\right)$.

    Then
    \begin{multline*}
        I_2^2 =\int_{\Omega_{r_0}^t}\bigg|\int_0^r(r-\tau)f(\tau)d\tau\bigg|^2 {v_\varphi^4\over r^2}\,\ud x\ud t' =\int_{\Omega_{r_0}^t}\bigg|\frac{1}{r}\int_0^r(r-\tau)f(\tau)d\tau\bigg|^2r^2v_\varphi^2{v_\varphi^2\over r^2}\,\ud x\ud t' \\
        \le\sup_{\Omega_{r_0}^t}|u|^2\int_{\Omega_{r_0}^t}\bigg|\frac{1}{r}\int_0^t(r-\tau)f(\tau)d\tau\bigg|^2{v_\varphi^2\over r^2}\,\ud x\ud t'\equiv L_1.
    \end{multline*}
    Using the H\"older inequality in $L_1$ implies
    \[
        L_1\leq \abs{u}_{\infty,\Omega_{r_0}^t}^2\int_0^t\left(\int_{\Omega_{r_0}}\abs{\frac{1}{r}\int_0^r(r-\tau)f(\tau)\, \ud\tau}^{2p}\ud x\right)^{2/2p}\, \ud t \sup_t \abs{u_1}_{2p',\Omega_{r_0}^T}^2\equiv L_2,
    \]
    where $1/p+1/p'=1$.

    Applying the Hardy inequality for the middle term in $L_2$, gives
    \begin{multline*}
        \int_0^t\left(\int_{\Omega_{r_0}}\bigg|\frac{1}{r}\int_0^r(r-\tau)f(\tau)d\tau\bigg|^{2p}\ud x\right)^{\frac{2}{2p}}\, \ud t' \\
        \le c\int_0^t\left(\int_{\Omega_{r_0}}\bigg|\int_0^rf(\tau)d\tau \bigg|^{2p}\ud x\right)^{\frac{2}{2p}}\, \ud t' \\
        \le c\int_0^t\left(\int_{\Omega_0}\bigg|\int_0^r(\psi_{1,\tau\tau}+ \psi_{1,\tau\tau}K(\tau))\, \ud\tau\bigg|^{2p}\ud x\right)^{\frac{2}{2p}}\, \ud t'\equiv L_3,
    \end{multline*}
    where we used that
    $$
    f=-\psi_{1,rr}\left(1 + K(r)\right).
    $$
    To apply the Hardy inequality we use the formula
    \[
        \int_0^r (r - \tau)f(\tau)\, \ud \tau = \int_0^r \int_0^\sigma f(\tau)\, \ud \tau\, \ud \sigma.
    \]
    Then, we use the following Hardy inequality (see e.g. \cite[Ch. 1, Sec. 2.16]{BIN})
    \[
        \left(\int_0^{r_0}\abs{\frac{1}{r}\int_0^r\int_0^\sigma f(\tau)\, \ud \tau\, \ud \sigma}^{2p}\, r\ud r\right)^{\frac{1}{2p}} \leq c \left(\int_0^{r_0}\abs{\int_0^r f(\tau)\, \ud \tau}^{2p}\, r\ud r\right)^{\frac{1}{2p}}
    \]
    Integrating the result with respect to $z$ we derive the first inequality in $L_3$.
    Continuing,
    $$
        L_3\le c\int_0^t\bigg(\int_{\Omega_{r_0}}\bigg(|\psi_{1,r}|^{2p}+ \bigg|\int_0^r\psi_{1,\tau\tau}K(\tau)d\tau\bigg|^{2p}\bigg)\,\ud x\bigg)^{2/2p}\,\ud t'\equiv L_4.
    $$
    Using
    $$
    \int_0^r\psi_{1,\tau\tau}K(\tau)d\tau=\psi_{1,r}K(r)-\int_0^r\psi_{1,\tau}K_{,\tau}d\tau
    $$
    in $L_4$ implies
    \begin{multline*}
        L_4 \le c\int_0^t\norm{\psi_{1,r}}_{2p,\Omega_{r_0}}^2\, \ud t' \\
        + \int_0^t\! \left(\int_{\Omega_{r_0}}\!\left(|\psi_{1,r}K(r)|^{2p} + \abs{\int_0^r\psi_{1,\tau} K_{,\tau}d\tau}^{2p}\right)\,\ud x\right)^{\frac{1}{p}} \\
        \leq c\int_0^t\norm{\psi_{1,r}}_{2p,\Omega_{r_0}}^2\, \ud t'\equiv L_5,
    \end{multline*}
    where the properties of $K$ are used. Finally, for $p\le 3$ and Lemma \ref{lem2.11}
    $$
    L_5\le c\abs{\omega_1}_{2,\Omega_{r_0}^t}^2
    $$
    Summarizing
    \begin{multline*}
        I_2^2\le c\abs{u}_{\infty,\Omega_{r_0}^t}^2\left(\abs{u_1(0)}_{2p',\Omega_{r_0}}^2 + \abs{f_1}_{2p',1,\Omega_{r_0}^t}^2\right)\\
        \cdot \exp\left(c\int_0^t\abs{\omega_1(t')}_{2,\Omega_{r_0}}^2 \,\ud t'\right)\abs{\omega_1}_{2,\Omega_{r_0}^t}^2,
    \end{multline*}
    where $p'\geq \frac{3}{2}$ and Lemma \ref{lem2.11} was used.

    Using estimates of $I_2^1$, $I_2^2$, $I_2^3$, we obtain
    \begin{multline*}
        I_2 \leq c\abs{u}_{\infty,\Omega_{r_0}^t}^2 \left(\abs{u}_{\infty,\Omega_{r_0}^t}^2 \abs{\nabla\omega_1}_{2,\Omega^t}^2 + \abs{\omega_1}_{2,\infty,\Omega^t}^4D_1^2 \vphantom{\left(\abs{u_1(0)}_{3,\Omega_{r_0}}^2 + \abs{f_1}_{3,1,\Omega_{r_0}^t}^2\right)}%
        \right.\\
        + \left.\left(\abs{u_1(0)}_{3,\Omega_{r_0}}^2 + \abs{f_1}_{3,1,\Omega_{r_0}^t}^2\right)\abs{\omega_1}_{2,\Omega^t}^2\exp\left(c\abs{\omega_1}_{2,\Omega^t}^2\right)\right).
    \end{multline*}
    Exploiting the estimate in the bound of $I$ we obtain \eqref{5.5}. This concludes the proof.
\end{proof}

\begin{proof}[Proof of Theorem \ref{t1}]
    Using \eqref{5.3} and estimates \eqref{5.4} and \eqref{5.5} in \eqref{5.1} and assuming that $\varepsilon_1$ and $\varepsilon_3$ are sufficiently small we obtain the inequality
\begin{multline}\label{5.11}
    \abs{\omega_1}_{2,\infty,\Omega^t}^2+\norm{\omega_1}_{L_2(0,t;H^1(\Omega)}^2\leq c\abs{u}_{\infty,\Omega_{r_0}^t}^2 \left(D_2^2\abs{\nabla\omega_1}_{2,\Omega^t}^2 + D_1^2\abs{\omega_1}_{2,\infty,\Omega^t}^4 \vphantom{\exp \left(c\abs{\omega_1}_{2,\Omega^t}^2\right)}%
    \right. \\
     + \left.\left(\abs{u_1(0)}_{3,\Omega_{r_0}}^2 + \abs{f_1}_{3,1,\Omega_{r_0}^t}^2\right)\abs{\omega_1}_{2,\Omega^t}^2\exp \left(c\abs{\omega_1}_{2,\Omega^t}^2\right)\right) + M(t),
\end{multline}
    where $M(t)$ is introduced in \eqref{1.21}.

    Let
    \begin{equation}
        X(t)=\abs{\omega_1}_{2,\infty,\Omega^t}^2+\norm{\omega_1}_{L_2(0,t;H^1(\Omega))}^2.
        \label{5.12}
    \end{equation}
    In view of this notation, \eqref{5.11} takes the form
    \begin{multline}
        X(t) \le c\abs{u}_{\infty,\Omega_{r_0}^t}^2\left(D_2^2X + D_1^2X^2 \vphantom{\left(\abs{u_1(0)}_{3,\Omega}^2 + \abs{f_1}_{3,1,\Omega_{r_0}^t}^2\right)}
    \right.\\
        + \left. X^2\exp\left(cX^2\right)\left(\abs{u_1(0)}_{3,\Omega}^2 + \abs{f_1}_{3,1,\Omega_{r_0}^t}^2\right)\right) + M(t) \equiv \epsilon F(X(t)) + M(t).
        \label{5.13}
    \end{multline}

    Consider the equality
    \begin{equation}\label{eq5.90}
        X'(t) = \epsilon F\left(X'(t)\right) + M(t).
    \end{equation}
    Using the method of successive approximations we will show that there exists a solution $X'(t)$ and determine the magnitude of $\epsilon$ which ensures the existence of this solutions.

    Suppose that
    \begin{equation}\label{eq5.100}
        X'_{n+1}(t) = \epsilon F\left(X_n(t)\right) + M(t).
    \end{equation}
    Let $\gamma > 1$. Recall that $M = M(T)$ and assume that
    \begin{equation}\label{eq5.110}
        \abs{X'_n(t)} \leq \gamma M.
    \end{equation}
    Then from \eqref{eq5.100} and \eqref{5.13} we obtain
    \begin{multline}
        \abs{X'_{n+1}(t)} \leq c \abs{u}_{\infty,\Omega^t} \left(D_2^2(\gamma M) + D_1^2(\gamma M)^2 \vphantom{\left(\abs{u_1(0)}^2_{3,\Omega} + \abs{f_1}^2_{3,1,\Omega_{r_0}^t}\right)}\right.\\%
        \left. + (\gamma M)^2 \exp\left(c(\gamma M)^2\right) \left(\abs{u_1(0)}^2_{3,\Omega} + \abs{f_1}^2_{3,1,\Omega_{r_0}^t}\right)\right) + M.
    \end{multline}
    Assume that
    \begin{multline*}
        \abs{u}_{\infty,\Omega^t)} \leq c(\gamma - 1)M \\
        \cdot \left(\gamma M D_2^2 + D_1^2 (\gamma M)^2 + (\gamma M)^2\exp\left(c(\gamma M)^2\right)\left(\abs{u_1(0)}^2_{3,\Omega} + \abs{f_1}^2_{3,1,\Omega_{r_0}^t}\right)\right)^{-1}.
    \end{multline*}
    Then
    \begin{equation}\label{eq5.130}
        \abs{X'_{n+1}(t)} \leq \gamma M.
    \end{equation}
    Let now $\omega_1(0)$ be given. Let $\tilde \omega_1$ be an extension of $\omega_1(0)$ such that $\abs{\tilde \omega_1}_{2,\infty(\Omega^t)}^2 + \norm{\tilde \omega_1}_{1,2,\Omega^t}^2 < \infty$ and $\tilde \omega_1\big\vert_{t = 0} = \omega_1(0)$.
    Let
    \begin{equation}\label{eq5.140}
        X_0' = \abs{\tilde \omega_1}_{2,\infty(\Omega^t)}^2 + \norm{\tilde \omega_1}_{1,2,\Omega^t}^2 < \gamma M
    \end{equation}
    Then, \eqref{eq5.110}, \eqref{eq5.130} and \eqref{eq5.140} imply that
    \[
        \abs{X'_n} \leq \gamma M \qquad \text{for all $n \in \mathbb{N}_0$}.
    \]

    It remains to check the convergence of $X_n'$. Let
    \[
        Y_n' = X_n' - X_{n-1}'.
    \]
    Then, \eqref{eq5.100} implies
    \begin{multline}
        Y_{n+1}' = c \abs{u}_{\infty,\Omega^t} \left(D_2^2 Y_n' + D_1^2\left(X_n'^2 - X_{n-1}'^2\right)^2 \vphantom{\left(\abs{u_1(0)}^2_{3,\Omega} + \abs{f_1}^2_{3,1,\Omega_{r_0}^t}\right)}\right.\\%
        \left. + \left(X_n'^2\exp\left(cX_n'^2\right) - X_{n+1}'^2 \exp\left(cX_{n-1}'^2\right)\right)\left(\abs{u_1(0)}^2_{3,\Omega} + \abs{f_1}^2_{3,1,\Omega_{r_0}^t}\right)\right).
    \end{multline}
    Continuing, we have
    \begin{multline*}
        \abs{Y_{n+1}'} \leq c \abs{u}_{\infty,\Omega^t} \left(D_2^2 \abs{Y_n'} + D_1^2\abs{Y_n'}\left(\abs{X_n'} + \abs{X_{n-1}'}\right) \vphantom{\left(\abs{u_1(0)}^2_{3,\Omega} + \abs{f_1}^2_{3,1,\Omega_{r_0}^t}\right)} + \left(\left(X_n'^2 - X_{n-1}'^2\right) \exp\left(cX_n'^2\right)\right.\right.\\%
        \left. + \left(X_n'^2 - X_{n-1}'^2\right) \exp\left(cX_{n-1}'^2\right) + X_{n-1}'^2\left(\exp\left(cX_n'^2\right) - \exp\left(cX_{n-1}'^2\right)\right)\right) \\
        \cdot \left.\left(\abs{u_1(0)}^2_{3,\Omega} + \abs{f_1}^2_{3,1,\Omega_{r_0}^t}\right)\right) \\
        \leq c \abs{u}_{\infty,\Omega^t} \left(D_2^2 \abs{Y_n'} + 2\gamma M D_1^2\abs{Y_n'} \vphantom{\left(\abs{u_1(0)}^2_{3,\Omega} + \abs{f_1}^2_{3,1,\Omega_{r_0}^t}\right)} + \left(\abs{Y_n'} 2\gamma M \exp\left(c\left(\gamma M\right)^2\right)\right.\right.\\%
        + \left.\left. \left(\gamma M\right)^2 \exp\left(c\left(\gamma M\right)^2\right) \abs{Y_n'} 2\gamma M\right) \left(\abs{u_1(0)}^2_{3,\Omega} + \abs{f_1}^2_{3,1,\Omega_{r_0}^t}\right)\right) \\
        = c \abs{u}_{\infty,\Omega^t} \left(D_2^2 + D_1^22\gamma M + 2\gamma M \left(\exp\left(c\left(\gamma M\right)^2\right)\right) + 2(\gamma M)^3\exp\left(c(\gamma M)^2\right)\right) \\
        \cdot \left(\abs{u_1(0)}^2_{3,\Omega} + \abs{f_1}^2_{3,1,\Omega_{r_0}^t}\right) \abs{Y_n'}
    \end{multline*}
    Hence, the sequence converges if
    \begin{multline*}
        \abs{u}_{\infty,\Omega^t} \left(D_2^2 + D_1^2(2\gamma M) + \left(2\gamma M \exp(c(\gamma M)^2)\right.\right. \\
        + \left.\left. 2(\gamma M)^3 \exp(c(\gamma M)^2)\right)\left(\abs{u_1(0)}^2_{3,\Omega} + \abs{f_1}^2_{3,1,\Omega_{r_0}^t}\right)\right) < 1.
    \end{multline*}
This ends the proof.
\end{proof}
As explained after Theorem \ref{t1} we have to emphasize that \eqref{1.25} is crucial for deducing the regularity of weak solutions to problem \eqref{1.5}.

\bibliographystyle{elsarticle-num}
\bibliography{bibliography}

\end{document}